\documentclass[3p,sort&compress]{personalartcl}
\pdfoutput=1
\usepackage[utf8]{inputenc}

\usepackage{amsmath,amscd,amsfonts,amssymb,amsthm}

\usepackage{graphicx}
\usepackage[usenames]{color}
\usepackage{tikz}

\usepackage{enumerate}
\usepackage{url}
\usepackage{mathrsfs} %Serve per \P (powerset)
\usetikzlibrary{matrix,arrows}
\usepackage{bbm} %For blackboard bold 1

\usepackage{verbatim,shortvrb}

\usepackage[colorlinks]{hyperref}
\hypersetup{%
colorlinks,
pdfborder= 1 0 1,
linkcolor=black, 
anchorcolor=black, 
citecolor=black, 
filecolor=black, 
menucolor=black, 
urlcolor=black
}

\newtheorem{theorem}{Theorem}[section]
\newtheorem{lemma}[theorem]{Lemma}
\newtheorem{corollary}[theorem]{Corollary}

\newtheorem*{claimnonum}{Claim}
\theoremstyle{definition}

\newtheorem{definition}[theorem]{Definition}
\newtheorem{remark}[theorem]{Remark}

\newcommand{\N}{\ensuremath{\mathbb{N}}}
\newcommand{\Z}{\ensuremath{\mathbb{Z}}}

\newcommand{\CC}{\ensuremath{\mathbb{C}}}
\newcommand{\V}{\ensuremath{\mathbb{V}}}
\renewcommand{\int}{\ensuremath{[0,1]}}

\DeclareMathOperator{\F}{\mathscr{F}} % FREE ALGEBRA
\DeclareMathOperator{\U}{\mathscr{E}} % Equivalence relation E

\DeclareMathOperator{\M}{\mathrm{M}} % Colimit

\renewcommand{\hom}{\operatorname{\mathsf{hom}}} % hom
\renewcommand{\lim}{\operatorname{\mathsf{lim}}} % limit

\newcommand{\diaA}{\ensuremath{(A_{i}, a_{ij}) _{i,j\in I}}}

\newcommand{\seqA}{\ensuremath{(A_{i}, a_{i}) _{i\in \N}}}  %%% SEQUENCES: WE CANONISE THE INDEX SET TO \omega
\newcommand{\seqB}{\ensuremath{(B_{k}, b_{k}) _{k\in \N}}}
\renewcommand{\a}{\overline{a}} %% colimit map
\renewcommand{\b}{\overline{b}} %% colimit map
\newcommand{\id}{\ensuremath{\textrm{id}}}

\newcommand{\norm}[1]{\left\lVert#1\right\rVert}

\newcommand{\restrict}[1]{{\raise-.7ex\hbox{\ensuremath{\mathbin\upharpoonright}}\raise-1.5ex\hbox{\scriptsize{$#1$}}}}

\renewcommand{\leq}{\leqslant}
\renewcommand{\geq}{\geqslant}

\newcommand{\C}{\ensuremath{\mathsf{C}}} % GENERAL CATEGORY C
\newcommand{\Cat}{\mathsf{C}} %generic category
\newcommand{\Set}{\mathsf{Set}} % Sets

\renewcommand{\U}{\mathscr{U}} %forgetful functor

\begin{document}
\begin{frontmatter}
\title{Two isomorphism criteria for directed colimits}%\tnoteref{t1}}
\author[milano]{Vincenzo Marra}
\ead{vincenzo.marra@unimi.it}
\author[salerno,adam]{Luca Spada}
\ead{lspada@unisa.it}
\address[milano]{Dipartimento di Matematica ``Federigo Enriques''. Università degli Studi di Milano. Via C. Saldini 50, 20133 Milano, Italy.}
\address[salerno]{Dipartimento di Matematica. Università degli Studi di Salerno. Via Giovanni Paolo II 132, 84084 Fisciano (SA), Italy.}
\address[adam]{ILLC, Universiteit van Amsterdam, Science Park 107, 1098XG Amsterdam, The Netherlands.}

%%% Keywords


\begin{abstract}
Using the general notions of finitely presentable and finitely generated object introduced by Gabriel and Ulmer in 1971, we prove that, in any (locally small) category, two sequences of finitely presentable objects and morphisms (or two sequences of finitely generated objects and monomorphisms) have isomorphic colimits (=direct limits) if, and only if, they are \emph{confluent}. The latter means  that the two given sequences  can be connected by a back-and-forth chain of morphisms that is cofinal on each side, and commutes with the sequences at each finite stage. In several concrete situations,  analogous isomorphism criteria are typically obtained by \emph{ad hoc} arguments.  The abstract results given here can  play the useful  r\^{o}le of discerning  the general from the specific in situations of actual interest. We illustrate by applying them to varieties  of algebras, on the one hand, and to \emph{dimension groups}---the ordered $K_{0}$ of approximately finite-dimensional  $C^{*}$-algebras---on the other. The first application encompasses such classical examples as Kurosh's isomorphism criterion for countable torsion-free Abelian groups of finite rank. The second application yields the Bratteli-Elliott  Isomorphism Criterion for dimension groups. Finally, we  discuss  Bratteli's original isomorphism criterion for approximately finite-dimensional $C^{*}$-algebras, and show that his result does not follow from ours.\end{abstract}

\end{frontmatter}
\section{Introduction, and statement of main result.}\label{s:intro}
In his pioneering work on Abelian categories, Gabriel investigated a categorial abstraction of the classical notion of Noetherian module \cite[Chapitre II.4]{gabriel}. This was one source of inspiration that later led Gabriel and Ulmer to define the concepts of finitely presentable and finitely generated object in any locally small category \cite[6.1]{gu}. (Cf.\ also the English summary in \cite{ulmer}.) Gabriel and Ulmer showed that, in an algebraic context, the two concepts agree with the standard notions of finitely presented and finitely generated algebra, respectively; see Lemma \ref{lem:var} below. The purpose of this note is to show that the Gabriel-Ulmer generalisation affords two widely applicable isomorphism criteria for directed colimits of sequences of objects and morphisms in an arbitrary category. In several concrete situations,  analogous isomorphism criteria are typically obtained by \emph{ad hoc} arguments.  The  results given here can thus play the   r\^{o}le of discerning  the general from the specific in situations of actual interest. To corroborate this  statement, in Section \ref{s:app} we shall give two applications of the  isomorphism criteria; we defer further remarks until then. Now we  turn to the statement of our main result, for which we recall a number of definitions.

Throughout, we let $\N:=\{1,2,3,\ldots\}$, and we let $\Set$ denote the category of sets and functions. Also, we always let $\Cat$ denote a locally small category, i.e.\ a category such that the class $\hom{(A,B)}$ of morphisms between any two given objects $A,B$ of $\Cat$ is a set. We  recall the standard notion of directed colimit (=filtered colimit over a directed index set) in a category;\footnote{While in category theory most authors work with filtered colimits, it is no loss of generality to restrict attention to directed colimits, as we shall do here after \cite{jacobson}. For a proof of the equivalence of the two notions, see e.g.\ \cite[Theorem 1.5]{ar}.} see e.g.\ \cite[Section 2.5]{jacobson}, where, following  a long-standing algebraic tradition, `directed colimits' are called `direct limits'.

 A set $I$ partially ordered by $\leq$ is (\emph{upward}) \emph{directed} if for any $i,j\in I$ there exists $k \in I$ with $i,j\leq k$. A \emph{directed diagram} in $\C$ is a pair $(B_i,b_{ij})$, where $i,j\in I$ and $i\leq j$, such that $B_{i}$ is a $\C$-object for each $i \in I$, and $b_{ij}\colon B_{i}\to B_{j}$ is a $\C$-arrow for each $i\leq j$. We call the $b_{ij}$'s the \emph{transition morphisms}.\footnote{Other names used in the literature include \emph{bonding} and \emph{connecting} maps.} 
 A \emph{cocone} for the diagram $(B_{i}, b_{ij})$  is a $\C$-object $B$ equipped with $\C$-arrows $\b_{i}\colon B_{i}\to B$,  one for each $i \in I$, that satisfy the commutativity relations $\b_{i}=\b_{j}\circ b_{ij}$, for each $i,j \in I$ with $i\leq j$. A \emph{colimit} in $\C$ of the  diagram  $(B_{i},b_{ij})$ is a \emph{universal cocone}  $(B,\b_{i})$: one such that, for any other cocone $(B',\b'_{i})$, there is a unique $\C$-arrow $f\colon B\to B'$ satisfying $\b'_{i}=f \circ \b_{i}$ for each $i \in I$. The $\C$-arrows $\b_{i}$ of the universal cocone are called the \emph{colimit morphisms}, and $B$ is the \emph{colimit object}. The universal property of course entails that directed colimits are unique to within a unique $\C$-isomorphism. 

Following \cite[6.1]{gu}, we say that a $\C$-object $A$  is \emph{finitely presentable} if the covariant hom-functor $\hom{(A,-)}$$\colon$ $\Cat \to \Set$ preserves directed colimits.
Explicitly, this means that if  $(B,\b_{i})$ is the colimit in $\C$ of the directed diagram $(B_{i},b_{ij})$,  then for every $\C$-arrow $f \colon A \to B$ the following two conditions are satisfied. (See Fig.~\ref{fig:fp}.)
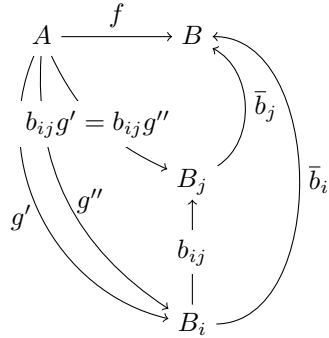
\begin{figure}[h!]\begin{center}
\begin{tikzpicture}[descr/.style={fill=white}]
\matrix(m)[matrix of math nodes, row sep=4em, column sep=4em, text height=1.5ex, text depth=0.25ex, ampersand replacement=\&]{
A\& B\\\& B_j\\\&B_i\\
};
\path[<-] (m-2-2) edge node[descr] {$b_{ij}$} (m-3-2);
\path[->] (m-1-1) edge node[above] {$f$} (m-1-2);
\path[->] (m-1-1) edge  [bend right=50] node[left] {$g'$} (m-3-2);
\path[->] (m-1-1) edge  [bend right=30] node[right] {$g''$} (m-3-2);
\path[->] (m-1-1) edge    [bend right=20] node[descr] {$b_{ij}g'=b_{ij}g''$} (m-2-2);
\path[->] (m-2-2) edge [bend right=60] node[right] {$\b_j$} (m-1-2);
\path[->] (m-3-2) edge [bend right=90] node[right] {$\b_i$} (m-1-2);
\end{tikzpicture}
\end{center}
\caption{Finite presentability according to Gabriel and Ulmer.}
\label{fig:fp}
\end{figure}
\begin{align}
& \text{There is } g \colon A \to B_i \text{ such that } f = \b_{i}\circ g.  \tag{F}\label{eq:factorization}\\
& \text{For any } g',g'' \colon A \to B_i \text{ such that } f = \b_{i}\circ g'=\b_{i}\circ g'', \text{ there is } j\geq i \text{ such that } b_{ij}\circ g'=b_{ij}\circ g''. \tag{E}\label{eq:essential-unique}
\end{align}
Here, (\ref{eq:factorization}) is known as the \emph{factorisation property}, and (\ref{eq:essential-unique}) as the \emph{essential uniqueness} property.
Again after \cite[6.1]{gu}, we say that $A$  is \emph{finitely generated} if  $\hom{(A,-)}\colon \Cat \to \Set$ preserves directed colimits of diagrams $(B_{i}, b_{ij})$ all of whose transition morphisms $b_{ij}$ are monomorphisms in $\C$.
\begin{remark}\label{rem:e}In the literature, the following  condition is often used  in place of (\ref{eq:essential-unique}):
\begin{align}
& \text{If }g' \colon A \to B_i \text{ is such that } f = \b_{i}\circ g', \text{ there is } j\geq i \text{ such that } b_{ij}\circ g=b_{ij}\circ g', \tag{E$'$}\label{eq:essential-uniquebis}
\end{align}
where $g \colon A \to B_i$ is the arrow whose existence is granted by \eqref{eq:factorization}.
Cf.\ e.g.\ \cite[p.\ 9]{ar}. But it is easy to check that $f \colon A \to B$ satisfies  (\ref{eq:factorization}) and (\ref{eq:essential-unique}) if, and only if, it satisfies (\ref{eq:factorization}) and (\ref{eq:essential-uniquebis}). It will transpire from Section \ref{s:proof} that    (\ref{eq:essential-unique}) is technically expedient  in this note, insofar as it does not depend on \eqref{eq:factorization}.
\end{remark}
We next specialise directed diagrams to sequences. By a \emph{sequence \textup{(}of objects and morphisms\textup{)}} in  $\C$ we mean  a pair $\seqA$ with $A_{i}$ a $\Cat$-object, and $a_{i}\colon A_{i}\to A_{i+1}$ a $\Cat$-arrow. For any two integers $0<i<j$ we define the $\Cat$-arrow $a_{ij}\colon A_{i}\to A_{j}$ as
\begin{align*}
a_{ij}:=a_{j-1}\circ\cdots\circ a_{i+1}\circ a_{i},
\end{align*}
and we also set $a_{ii}$ equal to the identity arrow on $A_{i}$.
By the \emph{colimit} in $\C$ of such a sequence we   mean the colimit of the directed diagram $(A_{i},a_{ij})$, with ${i,j\in \N}$.
In the central definition that follows we consider a second sequence $\seqB$ with associated directed diagram $(B_{k},b_{kl})$, ${k,l\in\N}$.
\begin{definition}[Confluent sequences]\label{d:confluence}We say that two sequences
 $\seqA$ and $\seqB$  in the category $\C$ are \emph{confluent} if there exist integers $0<i_{1}<i_{2}<\cdots$,  $0<k_{1}<k_{2}<\cdots$, together with $\C$-arrows $f_n\colon A_{i_{n}}\to B_{k_{n}}$ and $g_n\colon B_{k_{n}}\to A_{i_{n+1}}$ for each $n\in\N$, such that
 the commutativity relations
 \begin{align}
a_{i_{n}i_{n+1}}&=g_{n}\circ f_{n}  \label{eq:comm1} \\ 
b_{k_{n}k_{n+1}}&=f_{n+1}\circ g_{n} \label{eq:comm2}
 \end{align}
 hold. (See Fig.\ \ref{f:confluence}.) 

\end{definition}
\begin{figure}[h!]
\begin{center}
\begin{tikzpicture}[descr/.style={fill=white}]
\matrix(m)[matrix of math nodes, row sep=4em, column sep=3em, ampersand replacement=\&]{
{\ldots}		\& {A_{i_{n}}}		\&  {{A_{i}}} \&  {A_{i_{n+1}}}	\& {\ldots}\\
{\ldots}		\& {{B_{k}}}		\& {B_{k_{n}}} \& {B_{k_{n+1}}}	\& {\ldots}\\
};
\draw [->] (m-1-2) edge  node[descr] {$f_{n}$}  (m-2-3);
\draw [->] (m-2-3) edge  node[descr] {$g_{n}$}  (m-1-4);
\draw [->] (m-1-4) edge  node[descr] {$f_{n+1}$}  (m-2-4);
\draw [->] (m-1-2) edge node[above] {$a_{i_{n}}$} (m-1-3);
\draw [->] (m-1-3) edge node[above] {$a_{i}$} (m-1-4);
\path[->] (m-1-2) edge  [bend left=40] node[descr] {$a_{i_{n}i_{n+1}}$} (m-1-4);
\draw [->] (m-2-2) edge node[below] {$b_{k}$} (m-2-3);
\draw [->] (m-2-3) edge node[below] {$b_{k_{n}\,k_{n+1}}$} (m-2-4);
\end{tikzpicture}
\end{center}
\caption{A confluence between the sequences $\seqA$ and $\seqB$. (See Definition \ref{d:confluence}.)}
\label{f:confluence}
\end{figure}
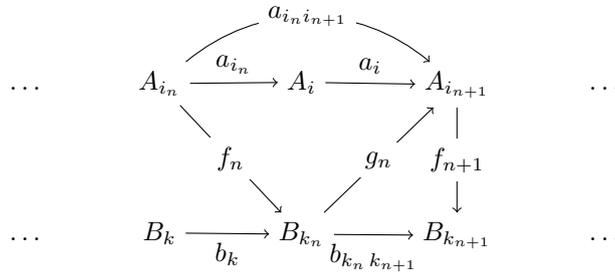

We shall see in Theorem \ref{t:confluence-implies-iso} below that, in full generality, confluence of sequences implies isomorphism of the colimit objects. The converse, however, fails even in the best behaved categories\footnote{We thank Benno van den Berg for suggesting the following stripped down example.}. In $\Set$, write $\N$ as a colimit of the sequence $\emptyset \hookrightarrow \{1\}\hookrightarrow \{1,2,\ldots\} \hookrightarrow \cdots$. Further write $\N$ as the colimit of the constant sequence $\N\rightarrow \N\rightarrow \cdots$, where each arrow is the identity function. Then it is clear that the sequences are not confluent. The second sequence does not consist of finitely presentable or finitely generated objects. Indeed, one checks that finitely presentable objects in $\Set$ coincide with the finitely generated ones, and they are precisely the finite sets. The main result of this note is that confluence is equivalent to isomorphism of the colimit objects under appropriate finiteness assumptions:
\begin{theorem}\label{t:confluence=iso}Let $\C$ be any locally small category.  Suppose two sequences $\seqA$ and $\seqB$ 
in $\C$  admit colimit objects $A$ and $B$ in $\C$, respectively. 
\begin{enumerate}
\item\label{t:confluence=iso:item1} Suppose that each $A_{i}$ and each $B_{k}$ is finitely presentable. Then $A$ and $B$ are isomorphic if, and only if, the two sequences are confluent.
\item\label{t:confluence=iso:item2}  Suppose that each $A_{i}$ and each $B_{k}$ is finitely generated, and that each $a_{i}$ and  each $b_{k}$ is a monomorphism. Then $A$ and $B$ are isomorphic if, and only if, the two sequences are confluent.
\end{enumerate}
\end{theorem}
Before turning to the proofs, we remark that extensions of Theorem \ref{t:confluence=iso} to more general directed diagrams are possible. Indeed, the case of all directed diagrams is closely related to Grothendieck's completion of a category under inductive systems (``ind-completion''), for which see \cite[VI and references therein]{johnstone}; cf.\ also Remark \ref{rem:ind} below. Such extensions, however, call for  generalised notions of confluence that are, as far as we can see, harder to apply. Since the case of sequences appears to be most common in concrete situations---see the examples in Section \ref{s:app}---we only deal with it in this note.
\section{Proof of Theorem \ref{t:confluence=iso}.}\label{s:proof}
\begin{theorem}\label{t:confluence-implies-iso}
Let $\seqA, \seqB$ be two sequences in $\C$ with colimits $(A,\a_{i})_{i\in \N}$ and $(B,\b_{k})_{k\in \N}$, respectively. If $\seqA, \seqB$  are confluent then $A$ and $B$ are isomorphic.
\end{theorem}
\begin{proof}
Let $0<i_{1}<i_{2}<\cdots$ and $0<k_{1}<k_{2}<\cdots$  be indices such that  $f_{n}\colon A_{i_{n}}\to B_{k_{n}}$ and $g_{n}\colon B_{k_{n}}\to A_{i_{n+1}}$, $n\in\N$,  are a confluence between the given sequences.  Consider the  subsequence $S:=(A_{i_{n}},a_{{i_n}i_{n+1}})_{n\in\N}$. A proof of the next claim may be obtained applying  standard technical results on final functors, see e.g.\ \cite[0.11]{ar} or \cite[VI.1.5]{johnstone}. We prefer to give a direct  proof here.
\begin{claimnonum}The colimit of \textup{(}the diagram associated with\textup{)} $S$ is $(A,\a_{i_{n}})_{n\in \N}$.
\end{claimnonum}
\begin{proof}[Proof of  Claim]
Clearly $(A,\a_{i_{n}})_{n\in \N}$ is a cocone for $S$. Let $(A',\a'_{i_{n}})_{n\in \N}$ be an arbitrary cocone for $S$.  Let us set $I:=\{i_{n}\}_{n\in \N}$ and $J:=\N\setminus I$. The poset $I$ is cofinal in $\N$, i.e.\ for any $j\in J$ there is $i\in I$ such that $i\geq j$. Let $\theta\colon J\to I$ be the function that associates with each $j\in J$ the minimum $i\in I$ with $i\geq j$. 
It is elementary to verify that 
\begin{align*}
K:=\left(A',\,\left\{\a'_{\theta(j)}\circ a_{j\theta(j)}\right\}_{j\in J}\bigcup\,\left\{\a'_{i_{n}}\right\}_{n\in\N}\right)
\end{align*}
is a cocone for $\seqA$. By the universal property of the colimit $(A,\a_{i})_{i\in \N}$, there exists a unique $\Cat$-arrow $h\colon A\to A'$ that factors all arrows in $K$ through $\{\a_{i}\}_{i\in \N}$.  Hence, \textit{a fortiori}, $h$ factors all $\{\a'_{i_{n}}\}_{n\in\N}$ through $\{\a_{i_{n}}\}_{n\in\N}$. It remains to show that $h$ is the unique $\Cat$-arrow that has the latter factorisation property.  If $k\colon A\to A'$ is another arrow such that $\a'_{i_{n}}=k\circ\a_{i_{n}}$ for every $n\in\N$, then $\a'_{\theta(j)}\circ a_{j\theta(j)}=k\circ\a_{\theta(j)}\circ a_{j\theta(j)}$ for $j\in J$. Upon noting that any $\a'_{i}$ is either of the form $\a'_{\theta(j)}$ for some $j \in J$, or else can be written in the form $\a'_{\theta(j)}\circ a_{j\theta(j)}$ for some $j \in J$, we infer that $k$ must be equal to $h$.
\end{proof}
Note now that $\left(B,\, b_{k_{n}}\circ f_{n}\right)_{n\in\N}$ is a cocone for the diagram associated with the sequence $S$.  For this it is sufficient to verify the commutativity of the  diagram below.
\begin{center}
\begin{tikzpicture}[descr/.style={fill=white}]
\matrix(m)[matrix of math nodes, row sep=3em, column sep=2.4em, ampersand replacement=\&]{
{A_{i_{n}}}		\& 	\& {A_{i_{n+1}}}	\& 	\&  \\
{B_{k_{n}}}		\& {\makebox{}}	\& {B_{k_{n+1}}}	\& {\ldots}	\&  {B}	\\
};
\draw [->] (m-1-1) edge  node[left] {$f_{n}$}  (m-2-1);
\draw [->] (m-2-1) edge  node[descr] {$g_{n}$}  (m-1-3);
\draw [->] (m-1-3) edge  node[right] {$f_{n+1}$}  (m-2-3);
\draw [->] (m-1-1) edge node[above] {$a_{i_{n},i_{n+1}}$} (m-1-3);
\draw [->] (m-2-1) edge node[below] {$b_{k_{n},k_{n+1}}$} (m-2-3);
\draw [->] (m-2-1) edge [bend right=30] node[below] {$\b_{k_{n}}$} (m-2-5);
\draw [->] (m-2-3) edge [bend right=20] node[below, xshift=-.9cm,yshift=.2cm] {$\b_{k_{n+1}}$} (m-2-5);
\end{tikzpicture}
\end{center}
We compute:
\begin{align*}
\b_{k_{n+1}}\circ f_{n+1}\circ a_{i_{n},i_{n+1}}&=\b_{k_{n+1}}\circ f_{n+1}\circ g_{n}\circ f_{n} & \text{by \eqref{eq:comm1}}  \\
&=\b_{k_{n+1}}\circ b_{k_{n},k_{n+1}}\circ f_{n}&\text{by \eqref{eq:comm2}}\\
&=\b_{k_{n}}\circ f_{n} & \text{by the commutativity of colimit}\\ &&\text{and transition morphisms.}
\end{align*}
Since, by the Claim above, $A$ is the universal cocone for the diagram associated with $S$, there must exist a unique  $f\colon A\to B$ that factors the family of arrows $\{\b_{k_{n}}\circ f_{n}\mid n\in \N\}$, i.e.\
\begin{align}
\b_{k_{n}}\circ f_{n}= f\circ \a_{i_{n}} \qquad\text{ for every }n\in \N.\label{eq:f-factors}
\end{align}
Arguing symmetrically, we obtain a unique $g\colon B\to A$ such that 
\begin{align}
\a_{i_{n+1}}\circ g_{n}= g\circ \b_{k_{n}} \qquad\text{ for every }n\in \N.\label{eq:g-factors}
\end{align}
Now, to prove that the composition $g\circ f$ is the identity on $A$, we compute for every $n\in\N$:
\begin{align*}
g\circ f\circ \a_{i_{n}}&= g\circ \b_{k_{n}}\circ f_{n} & \text{by \eqref{eq:f-factors}}\\
	&=\a_{i_{n+1}}\circ g_{n}\circ f_{n} & \text{by \eqref{eq:g-factors}}\\
	&=\a_{i_{n+1}}\circ a_{i_{n}i_{n+1}}& \text{by \eqref{eq:comm1}}\\
	&=\a_{i_{n}}.& 
\end{align*}
So $g\circ f$ factors all colimit morphisms of the colimit $(A,\a_{i_{n}})_{n\in \N}$ of $S$. But, by the universal property of colimits, there exists a unique morphism with this property, and it is the identity---hence $g\circ f=\id_{A}$.  The proof that  $f\circ g=\id_{B}$ is  symmetrical. Therefore $f$ and $g$ are isomorphisms, and the proof is complete. 
\end{proof}

The next technical lemma establishes a key finiteness property of directed colimits of diagrams all of whose objects satisfy condition \eqref{eq:essential-unique} with respect to the colimit. It asserts that any arrow $p\colon A_{i}\to A_{j}$ that equalises two colimit arrows must already equalise two transition maps. We will use this in the proof of  Lemma \ref{l:local} below.
\begin{lemma}\label{l:weak-implies-eventual}
Let $\diaA$ be a directed diagram in $\C$, and let $(A, \a_{i})_{i\in I}$ be its colimit.  Suppose that  every $\C$-arrow $f\colon A_{i}\to A$ satisfies \eqref{eq:essential-unique} w.r.t. $\diaA$.  Consider any \C-arrow $p\colon A_{i}\to A_{j}$ with $i\leq j$. If $\a_{i}=\a_{j}\circ p$, then there exists $i_{0}\geq j$ such that $a_{ji_{0}}\circ p=a_{ii_{0}}$.
\end{lemma}
\begin{proof}
The colimit morphism $\a_{i}\colon A_{i}\to A$ factors through each transition morphism $a_{ij}$ with $i\leq j$, in symbols,
\begin{align}\label{eq:one}
\a_{i}=\a_{j}\circ a_{ij}.
\end{align}
By hypothesis:
 \begin{align}\label{eq:two}
 \a_{i}=\a_{j}\circ p.
 \end{align}
 Applying   \eqref{eq:essential-unique} to (\ref{eq:one}--\ref{eq:two}) with $f:=\a_{i}$, $g':=a_{ij}$, and $g'':=p$, we infer that there must be $i_{0}\geq j$ such that $a_{ji_{0}}\circ p=a_{ji_{0}}\circ a_{ij}=a_{ii_{0}}$.
\end{proof}
\begin{lemma}[Local isomorphism criterion]\label{l:local}
Let $\seqA, \seqB$ be two sequences in $\C$ with colimit objects  $(A,\a_{i})_{i\in\N}$ and $(B,\b_{k})_{k\in\N}$, respectively. Suppose that $A$ and $B$ are isomorphic. Suppose further that the following hold for each $i,k\in \N$.
\begin{enumerate}[{\rm (H1)}]
\item\label{l:local:item1} Every $\C$-arrow $A_{i}\to B$ satisfies \eqref{eq:factorization} w.r.t. $\seqB$.
\item\label{l:local:item2} Every $\C$-arrow $A_{i}\to A$ satisfies \eqref{eq:essential-unique} w.r.t. $\seqA$.  
\item\label{l:local:item3} Every $\C$-arrow $B_{k}\to A$ satisfies \eqref{eq:factorization} w.r.t. $\seqA$.
\item\label{l:local:item4} Every $\C$-arrow $B_{k}\to B$ satisfies \eqref{eq:essential-unique} w.r.t. $\seqB$.  
\end{enumerate}
Then $\seqA$ and  $\seqB$  are confluent.
\end{lemma}
\begin{proof}
Let $f\colon A\to B$ be an isomorphism. By induction on $n\in \N$, we define   sequences of integers $0<i_{1}<\cdots <i_{n}<\cdots$ and $0<k_{1}<\cdots<k_{n}<\cdots$, and morphisms
\begin{align}
f_{n}&\colon A_{i_{n}}\to B_{k_{n}}\label{eq:fn}\\
g_{n}&\colon B_{k_{n}}\to A_{i_{n+1}},\label{eq:gn}
\end{align}
such that the following equalities  hold: 
\begin{align}
f\circ \a_{i_{n}}&=\b_{k_{n}}\circ f_{n}\label{eq:ind1}\\
f^{-1}\circ \b_{k_{n}}&=\a_{i_{n+1}}\circ g_{n}\label{eq:ind2}.
\end{align}  
For $n=1$, we set $i_{1}=1$. By (H\ref{l:local:item1}), there exists an arrow $f_{1}\colon A_{i_{1}}\to B_{k}$, for some $k \in \N$, that factors $f\circ \a_{i_{1}}$. Let us set $k_{1}:=k$; hence  $f\circ \a_{i_{1}}=\b_{k_{1}}\circ f_{1}$, and \eqref{eq:ind1} is satisfied.  Next, by  (H\ref{l:local:item3}) there  exists an arrow $g_{1}\colon B_{k_{1}}\to A_{i}$, for some $i \in \N$, that factors $f^{-1}\circ \b_{k_{1}}$.  Set $i_{2}:=i$; then  $f^{-1}\circ \b_{k_{1}}=\a_{i_{2}}\circ g_{1}$, and \eqref{eq:ind2} holds, too.

For the inductive step, assume that such arrows $f_{u},g_{u}$ have been defined for $1\leq u\leq n-1$.
Since, by (H\ref{l:local:item1}),  $f\circ \a_{i_{n}}\colon A_{i_{n}}\to B$ satisfies \eqref{eq:factorization} with respect to  $(B_{k},b_{k})$, there exist $l\in \N$ and $d_{n}\colon A_{i_{n}}\to B_{l}$ such that 
\begin{align}
f\circ \a_{i_{n}}=\b_{l}\circ d_{n}\label{eq:GU-confluenceA}.
\end{align}
We may safely assume that $l> k_{n-1}$.  For, in the contrary case, it suffices to  replace $d_{n}$ with its composition with a transition arrow $b_{ll'}$ for some $l'>k_{n-1}$, and to rewrite the right-hand side of \eqref{eq:GU-confluenceA} as $\b_{l}\circ d_{n}=\b_{l'}\circ b_{ll'}\circ d_{n}$.  By the induction hypothesis,  $g_{n-1}\colon B_{k_{n-1}}\to A_{i_{n}}$ satisfies the following equality:
\begin{align}
f^{-1}\circ \b_{k_{n-1}}=\a_{i_{n}}\circ g_{n-1}\label{eq:GU-confluenceB}.
\end{align}
Composing both sides of \eqref{eq:GU-confluenceA}  on the right with $g_{n-1}$ we obtain $f\circ \a_{i_{n}}\circ g_{n-1}=\b_{l}\circ d_{n}\circ g_{n-1}$. Using \eqref{eq:GU-confluenceB}, the latter equation yields $f\circ f^{-1}\circ \b_{k_{n-1}} =\b_{l}\circ d_{n}\circ g_{n-1}$, and hence
\begin{align}\label{eq:eventual}
\b_{k_{n-1}} =\b_{l}\circ d_{n}\circ g_{n-1}.
\end{align}
In light of (H\ref{l:local:item2}), we may apply Lemma \ref{l:weak-implies-eventual} to \eqref{eq:eventual} (with $p= d_{n}\circ g_{n-1}$) to deduce that there exists $k'\in\N$, $k'\geq l$, such that
\begin{align}
b_{k_{n-1}k'}=b_{l,k'}\circ d_{n}\circ g_{n-1}.\label{eq:ev-commutativity-f-g}
\end{align}  
Upon setting   $k_{n}:=k'$ and $f_{n}:=b_{lk_{n}}\circ d_{n}$, \eqref{eq:ev-commutativity-f-g} reads
\begin{align}
b_{k_{n-1}k_{n}}=f_{n}\circ g_{n-1},\label{eq:commutativity-f-g}
\end{align}  
so that \eqref{eq:comm2} in Definition \ref{d:confluence} is satisfied.  To conclude this part of the inductive step, we show that  $f_{n}$ satisfies \eqref{eq:ind1}:
\begin{align*}
f\circ \a_{i_{n}}&=\b_{l}\circ d_{n}& \text{by }\eqref{eq:GU-confluenceA},\\
&=\b_{k_{n}}\circ b_{lk_{n}}\circ d_{n}& \text{by the commutativity of tran-}\\&&\text{sition and colimit morphisms,}\\
&=\b_{k_{n}}\circ f_{n}& \text{by the definition of } f_{n}.
\end{align*}
The inductive construction of the arrow $g_{n}\colon B_{k_{n}}\to A_{i_{n+1}}$, and the verification that \eqref{eq:comm1} in Definition \ref{d:confluence} holds, are symmetric to the above. The proof is complete.
\end{proof}
\begin{proof}[End of Proof of Theorem \ref{t:confluence=iso}]
The right-to-left implications in parts 1 and 2 of the statement are given by Theorem \ref{t:confluence-implies-iso}. For the left-to-right implications, it is enough to check that all hypotheses in Lemma \ref{l:local} are satisfied. 
Concerning part 1, by assumption each $A_{i}$ is finitely presentable,   $B$ is the colimit of $\seqB$, and $A$ is the colimit of $\seqA$; hence every $\C$-arrow $A_{i}\to B$ satisfies \eqref{eq:factorization} w.r.t.\ $\seqB$, and every $\C$-arrow $A_{i}\to A$ satisfies \eqref{eq:essential-unique} w.r.t.\ $\seqA$.  Conditions (H\ref{l:local:item3}--H\ref{l:local:item4}) in Lemma \ref{l:local} hold by a symmetric argument.  The fact that Lemma \ref{l:local} applies under the assumptions of part 2 is equally straightforward. The theorem is proved.
\end{proof}

\section{Applications.}\label{s:app}
\subsection{Varieties of algebras.}\label{ss:var} For background on the universal-algebraic notions used in this subsection we refer to \cite[Chapter 2]{jacobson}. By a \emph{variety of algebras} \cite[Section 2.8]{jacobson} we mean, as usual, the class of all algebraic structures of the same type that satisfy some set of equations; equivalently, by Birkhoff's Variety Theorem  \cite[Theorem 2.14]{jacobson}, a class that is closed under homomorphic images, subalgebras, and direct products. We regard varieties as  categories in the obvious manner, with the homomorphisms as morphisms. If $\V$ is any variety, there is a forgetful functor $\U\colon\V\to\Set$ that carries algebras to underlying sets, and homomorphisms to functions. Then $\U$ has a left adjoint $\F\colon \Set \to \V$ called the \emph{free functor}, and the algebra $\F{(S)}$ is called the \emph{free algebra} (in $\V$) \emph{generated by the set $S$}. Free algebras are characterised to within isomorphism by the usual universal mapping property; see \cite[Theorem 2.9]{jacobson}. Let us write $\F_n$ for $\F{(S)}$, $S$ of finite cardinality $n\geq 0$.  A \emph{congruence} on an algebra $A$ in $\V$ is an equivalence relation on $A$ that is also a subalgebra of $A\times A$. By the first isomorphism theorem \cite[p. 62]{jacobson}, congruences on $A$ are in bijection with  kernels of onto homomorphisms with domain $A$. An algebra in $\V$ is \emph{finitely generated} if it is isomorphic to a quotient of $\F_{n}$, for $n\geq 0$ an integer. A congruence on $A$ is \emph{finitely generated} if it is the intersection of all congruences containing some finite subset of $A\times A$. Adapting a standard notion from group theory, an algebra in $\V$ is   \emph{finitely presented} if it is (isomorphic to) a quotient of $\F_{n}$ modulo some finitely generated congruence, for some integer $n\geq 0$. It is a standard fact that varieties of algebras are closed under directed colimits \cite[Theorem 2.8]{jacobson}.
\begin{lemma}\label{lem:var}In any  variety of algebras, an algebra is finitely presentable \textup{(}respectively, finitely generated\textup{)} in the sense of Gabriel and Ulmer if, and only if, it is finitely presented \textup{(}respectively, finitely generated\textup{)} in the usual algebraic sense.
\end{lemma}
\begin{proof}This was originally proved in \cite[7.6 or 9.3; see also p.\ 64]{gu}. It is also proved more directly in \cite[Theorem 3.12 and Proposition 3.11]{ar}, and in \cite[Proposition VI.2.2]{johnstone}. 
\end{proof}
\noindent It is well known that in any variety of algebras each algebra is representable either as a directed colimit of finitely presentable algebras (see e.g. \cite[Lemma VI.2.2]{johnstone}), or as a directed colimit of its finitely generated subalgebras (see e.g. \cite[Theorem 2.7]{jacobson}) with inclusion maps as transition homomorphisms. We omit the  proof of the fact that one can specialise these results as follows.
\begin{enumerate}[(I)]
\item\label{eq:app1} Each finitely generated algebra in a variety is the directed colimit of a sequence of finitely presented algebras and surjective transition homomorphisms.
\item\label{eq:app2} Each countable algebra in a variety is the directed colimit of a sequence of finitely generated subalgebras and injective transition homomorphisms.
\end{enumerate}
Facts (\ref{eq:app1}--\ref{eq:app2}) should clarify the scope of:
\begin{corollary}\label{cor:var}In any variety of algebras, suppose two algebras $A$ and $B$ are the directed colimits of  the sequences $\seqA$ and $\seqB$, respectively.  
\begin{enumerate}
\item\label{cor:var:item1} Suppose that each $A_{i}$ and each $B_{k}$ is a finitely presented algebra. Then $A$ and $B$ are isomorphic if, and only if, the two sequences are confluent.
\item\label{cor:var:item2}  Suppose that each $A_{i}$ and each $B_{k}$ is a finitely generated algebra, and that each $a_{i}$ and  each $b_{k}$ is an injective homomorphisms. Then $A$ and $B$ are isomorphic if, and only if, the two sequences are confluent.
\end{enumerate}
\end{corollary}
\begin{proof}Apply Theorem \ref{t:confluence=iso} and Lemma \ref{lem:var}, together with the easy observation that a homomorphism between algebras is a monomorphism if, and only if, it is an injective function \cite[Proposition 2.3]{jacobson}.\end{proof}
\begin{remark}In a number of classical papers, Derry, Kurosh, and Mal'cev developed a technique to construct complete isomorphisms invariants for countable torsion-free Abelian groups of finite rank in terms of equivalence classes of families of matrices. The hypothesis of finite rank was later removed by Fuchs; we refer to his textbook account \cite[\S 45]{fuchs} for details and references.\footnote{It has long been recognised that these results can hardly be considered as a usable classification of countable torsion-free Abelian groups, because the equivalence relation in question is as complicated as the original isomorphism problem. See Fuchs' remarks on \cite[p.\ 157]{fuchs}.} We sketch here a proof of the fact that the existence of invariants of this sort essentially follows from Corollary \ref{cor:var}. Indeed,  we can represent each countable torsion-free Abelian group as the directed colimit of some sequence
\begin{align}\label{eq:groups}
\Z^{r_{1}}\overset{\iota_{1}}{\longrightarrow} \Z^{r_{2}}\overset{\iota_{2}}\longrightarrow\cdots \Z^{r_{i}}\overset{\iota_{i}}{\longrightarrow} \cdots
\end{align}
with each $\iota_{i}$ an injective group homomorphism. The  approximation \eqref{eq:groups} is a special case of \eqref{eq:app2}, up to  the Fundamental Theorem of Abelian groups that a finitely generated torsion-free Abelian group is free of finite rank.  Corollary \ref{cor:var}\eqref{cor:var:item2} now tells us that two sequences as in (\ref{eq:groups}) have isomorphic groups as colimit objects if, and only if, they are confluent. Choosing $\Z$-module bases,  representing each group homomorphism featuring in a confluence diagram as a matrix with integer entries,   and writing down the commutativity conditions (\ref{eq:comm1}--\ref{eq:comm2}) in terms of products of such matrices, one obtains an equivalence relation on sequences of matrices that  yields a complete isomorphism invariant closely related to the classical ones referred to above.
\end{remark}
\begin{remark}In the recent paper \cite{BusCabMun},  Busaniche, Cabrer, and Mundici prove an isomorphism criterion for directed colimits of sequences of finitely presented lattice-ordered Abelian groups with an order-unit, where all transition morphisms are assumed to be surjective. Their proofs make use of the non-trivial representation of these structures in terms of  compact polyhedra and piecewise linear maps. (For further information on the topic, the interested reader may consult \cite[and references therein]{marra_mundici, marspa12, marraspadaAPAL}.) Nonetheless, their isomorphism criterion \cite[Theorem 3.3]{BusCabMun} is an immediate consequence of Corollary \ref{cor:var}\eqref{cor:var:item1}, given that by a well-known theorem of Mundici \cite[Theorem 3.9]{Mundici86} the category in question is equivalent to a variety of algebras, namely, Chang's MV-algebras \cite{cdm}. In fact, our results show that the assumption in  \cite[Theorem 3.3]{BusCabMun} that all transition morphisms be surjective may be dropped.
\end{remark}
\subsection{The Bratteli-Elliott Isomorphism Criterion for dimension groups, and  AF $C^{*}$-algebras.}\label{ss:bratteli} For background on dimension groups see \cite{effros, goodearl}. A \emph{partially ordered Abelian group} is an Abelian group $G$ (which we always write additively) equipped with a \emph{translation-invariant} 
partial order: $x\leq y$ implies $x+t\leq y+t$ for each $x,y,t\in G$.  A morphism of partially ordered Abelian groups is a group homomorphism $G\to H$ that is order-preserving. The \emph{positive cone} of $G$ is $G^{+}:=\{x\in G \mid x \geq 0\}$. One says that $G$ is \emph{directed} if it is generated (as a group) by $G^{+}$; and that it is \emph{unperforated} (or has \emph{isolated order}) if $nx \geq 0$ for some $x \in G$ and $n \in \N$ implies $x \geq 0$. Further, $G$ has (\emph{Riesz}) \emph{interpolation} if whenever $x_1, x_2, y_1, y_2 \in G$ satisfy  $x_{i} \leq y_j$, there exists $z \in G$ such that $x_i\leq z\leq y_{j}$.  A \emph{dimension group} is any partially ordered Abelian group that is directed, unperforated, and has interpolation. The \emph{category of dimension groups} is the full subcategory of the category of partially ordered Abelian groups whose objects are dimension groups. A \emph{simplicial group} is (any partially ordered Abelian group isomorphic to) the free Abelian group $\Z^{r}$ of rank $r \geq 0$, equipped with the pointwise order inherited from $\Z$, i.e.\ with the positive cone $(\Z^{r})^{+}:=\{(z_1,\ldots,z_{r})\mid z_i \geq 0 \text{ for } i=1,\ldots,r\}$. It is clear that each simplicial group is a dimension group. The category of dimension groups has directed colimits, and they are computed as directed colimits of Abelian groups, where the colimit group is equipped with the positive cone obtained as the set-theoretic colimit of the positive cones in the diagram; see  \cite[1.15]{goodearl}. Specifically, consider a sequence 
\begin{align}\label{eq:dimgrp}
\Z^{r_{1}}\overset{h_{1}}{\longrightarrow} \Z^{r_{2}}\overset{h_{2}}\longrightarrow\cdots \Z^{r_{i}}\overset{h_{i}}{\longrightarrow} \cdots
\end{align}
of simplicial groups and order-preserving homomorphisms. If $D$ is the colimit of (\ref{eq:dimgrp}) in the category of Abelian groups, and $\overline{h}_{i}\colon \Z^{r_{i}}\to D$ are the colimit maps,  then setting $D^{+}:=\bigcup_{i\in\N}\overline{h}_i((\Z^{r_i})^{+})$ yields a positive cone on $D$ that makes it into a partially ordered Abelian group. It is elementary to check that, with this order, $D$ is in fact a dimension group. 
\begin{lemma}\label{lem:fpsimplicial}In the category of dimension groups, each simplicial group is finitely presentable in the sense of Gabriel and Ulmer.
\end{lemma}
\begin{proof}Let $A$ be a dimension group that is the colimit object of a directed diagram $\diaA$ of dimension groups and order-preserving group homomorphisms, and write $\a_{i}\colon A_{i}\to A$ for the colimit morphisms. Consider an order-preserving group homomorphism $f\colon\Z^{n}\to A$, for some integer $n \in \N$. Let $e_{1},\ldots,e_{n}$ denote the  standard $\Z$-module basis of $\Z^{n}$. By \cite[1.15]{goodearl}, for each $k=1,\ldots,n$ there exists $i_{k}\in I$ and $e'_{k} \in A_{i_{k}}^{+}$ such that $a_{i_{k}}(e'_{k})=f(e_{k})$. By the directedness of the index set, there exist $i_{0}$ and $e'_{1},\ldots,e'_{n}\in A_{i_{0}}^{+}$ such that  $a_{i_{k}}(e'_{k})=f(e_{k})$ holds for each $k=1,\ldots,n$. Since simplicial groups are  free Abelian groups, there exists a unique group homomorphism $g\colon \Z^{n}\to A_{i_{0}}$ that satisfies $f(e_{i})=e'_{i}$, for each $k=1,\ldots,n$. Clearly $f=\a_{i_{0}}\circ g$. Note that $g$ preserves the positive cone of $\Z^{n}$ (i.e.\ is \emph{positive}), and therefore it is order-preserving. This proves (\ref{eq:factorization}). It is also clear that $g$ is essentially unique in the sense of (\ref{eq:essential-unique}), because any two group homomorphisms $g',g''\colon \Z^{n}\to A_{i}$ are uniquely determined by their action on $e_{1},\ldots,e_{n}$.
\end{proof}
\begin{corollary}[The Bratteli-Elliott Isomorphism Criterion]\label{cor:elliottbratteli}Two sequences of simplicial groups and order-preserving group homomorphisms as in \textup{(\ref{eq:dimgrp})} have isomorphic colimit dimension groups if, and only if, they are confluent. 
\end{corollary}
\begin{proof} Lemma \ref{lem:fpsimplicial} and Theorem \ref{t:confluence=iso}\eqref{t:confluence=iso:item1}.
\end{proof}
\begin{remark}\label{rem:ind}The converse of Lemma \ref{lem:fpsimplicial} is also true, though not as easy to prove: it can be shown that it follows from the important Effros-Handelman-Shen Theorem \cite[Theorem 2.2]{EHS} that each dimension group is the colimit object of a directed diagram of simplicial groups,  the colimit being taken in the category of partially ordered Abelian groups. Thus, the characteristic approximation property of dimension groups is a notable instance of the much-studied abstract theory of accessing objects in a category via directed colimits of finitely presentable objects \cite{gu, ar}. In this connection, an obvious question for further research  is whether dimension groups are the ind-completion of the category of simplicial groups. We plan to pursue this question elsewhere, and for now refer the interested reader  to \cite{gw}, where related issues are considered from the perspective of universal algebra. \end{remark}
We close this paper with some remarks on AF  $C^{*}$ algebras; see e.g.\ \cite{goodearlnotes, Rordametal} for background. Recall that a \emph{$C^{*}$-algebra} is a (not necessarily commutative) algebra over $\CC$, equipped with a norm $\norm{\cdot}$ and an involution $*$, such that $A$ is complete with respect to $\norm{\cdot}$, and satisfies $\norm{ab}\leq\norm{a}\norm{b}$ and $\norm{a^*a}=\norm{a}^{2}$ for all $a,b\in A$. We always  assume $C^{*}$-algebras to be \emph{unital}, i.e.\ to come with a multiplicative identity.\footnote{Although  dimension groups in the above were not equipped with an order unit, we prefer to discuss unital $C^{*}$-algebras here to simplify references to the literature.}  Morphisms of $C^{*}$-algebras are taken to be the \emph{$*$-homomorphisms}, i.e.\ the (unit-preserving) algebra morphisms that commute with the involution. It is proved in \cite[16.3]{goodearl} that arbitrary directed colimits exist in the category of $C^{*}$-algebras.
We write $\M_n:=\M_{n}{(\CC)}$ for the set of all $n\times n$ matrices with complex entries. Then $\M_n$ is a $C^*$-algebra with the involution and the norm obtained from identifying it with the bounded linear operators on the Hilbert space $\CC^{n}$. In particular, the involution is the conjugate transpose involution that takes a matrix $(x_{ij})$ to $(\overline{x}_{ji})$, where $\overline{x}$ is the conjugate of $x \in \CC$. The $C^{*}$-algebras $\M_{n_{1}}\times \cdots  \times \M_{n_{t}}$, for integers $n_{i},t\geq 1$, are defined on the product algebra using componentwise operations and the supremum norm. By a standard result (\cite[1.5]{goodearlnotes}, \cite[7.1.5]{Rordametal}), up to an isomorphism the non-trivial $C^{*}$-algebras that are finite-dimensional (as algebras) are isomorphic to $\M_{n_{1}}\times \cdots  \times \M_{n_{t}}$ for some choice of the integers $n_{i}$. In a seminal paper, Bratteli \cite[1.1]{bratteli} defined a $C^{*}$-algebra $A$ to be  \emph{approximately finite-dimensional} (or an \emph{AF $C^*$-algebra}, for short) if it satisfies $A=\overline{\bigcup_{i\in\N}A_{i}}$ for some inclusion-increasing subsequence of finite-dimensional subalgebras $A_{i}$, where $\overline{\,\cdot\,}$ denotes the norm-closure. It is known \cite[16.3]{goodearlnotes} that this is the same as defining AF $C^{*}$-algebras as the directed colimits of sequences of  finite-dimensional $C^{*}$-algebras (with arbitrary transition $*$-homomorphisms).

The \emph{Bratteli Isomorphism Criterion}  \cite[2.7]{bratteli} asserts that two AF $C^{*}$-algebras $A=\overline{\bigcup_{i\in\N}A_{i}}$ and $B=\overline{\bigcup_{j\in\N}B_{j}}$ are isomorphic if, and only if, the two approximating sequences $\{A_{i}\}$ and $\{B_{j}\}$ of finite-dimensional $C^{*}$-algebras and inclusions are confluent in the sense of Definition \ref{d:confluence}. It is important to emphasise that \emph{Bratteli's criterion does not follow from our Theorem \ref{t:confluence=iso}.} This is because in the category of $C^{*}$-algebras (or even in the full subcategory of AF $C^{*}$-algebras) a finite-dimensional algebra is  in general neither  finitely generated nor finitely presentable in the sense of Gabriel and Ulmer. Indeed, consider a $*$-homomorphism $h\colon C\to \overline{\bigcup_{i\in\N}A_{i}}$, with $C$ finite-dimensional, and pick a set of matrix units $\{e_{pq}^{t}\}$ in $C$, for integers $1\leq t \leq r$ and $1\leq p,q\leq n_{t}$. (See \cite[1.5]{bratteli}, \cite[16]{goodearlnotes}, and \cite[7.1.1]{Rordametal} for background on matrix units.) Writing $\a_{k}\colon A_{k}\to A$ for the   colimit $*$-homomorphisms, there \emph{need not} exist an integer $k \in \N$ such that the images $h(e_{pq}^{t})$ of the chosen matrix units are contained in $A_{k}$:  by Bratteli's intrinsic characterisation of AF $C^{*}$-algebras \cite[2.2]{bratteli}, the most that can be said is that there exist an integer $k\in \N$, elements $f_{pq}^{t}$ in $A_{k}$, and a real number $\epsilon > 0$,   such that $\|f_{pq}^{t}-h(e_{pq}^{t})\|<\epsilon$. But  by the universal property of matrix units \cite[p.\ 110]{Rordametal}, the required essentially unique factorisation of $h$ only exists if the stronger containment condition $\{h(e_{pq}^{t})\}\subseteq A_{k}$ can be satisfied for some $k\in \N$. The problem here is that, unlike the situation for varieties of algebras, directed colimits in the category of $C^{*}$-algebras are not computed as the set-theoretic colimit with additional structure; see \cite[16.2]{goodearlnotes}. Even though dimension groups are not defined as a variety of algebras, their directed colimits are computed as in sets, as we recalled above. This offers a genuine simplification with respect to AF $C^{*}$-algebras, and explains why upon passing from AF $C^{*}$-algebras to dimension groups by taking Elliott's ordered $K_{0}$ (see \cite[7.3.4]{Rordametal}) one can obtain the easy proof above of Corollary \ref{cor:elliottbratteli}.

\smallskip
\noindent {\bf Acknowledgments.}  Both authors acknowledge  partial support from the Italian National Research Project  2010FP79LR.  Parts of this article were written while the authors were kindly hosted by the Argentinian National Research Council (CONICET) in Buenos Aires within the European FP7-IRSES project \emph{MaToMUVI} (PIEF-GA-2009-247584).  The second-named author gratefully acknowledges  partial support by a Marie Curie Intra-European Fellowship within the project \emph{ADAMS} (PIEF-GA-2011-299071). The first-named author acknowledges  partial support by the Italian FIRB ``Futuro in Ricerca'' grant RBFR10DGUA.

\end{document}